\documentclass{amsart}

\usepackage{amssymb}
\usepackage{amsmath}
\usepackage{mathrsfs}
\usepackage{dsfont}
\usepackage{url}

\newtheorem{theorem}{Theorem}
\newtheorem*{theorem*}{Theorem}
\newtheorem{lemma}{Lemma}
\newtheorem*{lemma*}{Lemma}
\numberwithin{equation}{section}

\begin{document}

\title[The Average Number of Goldbach Representations ]{The Average Number of Goldbach Representations}

\author[D. A. Goldston]{D. A. Goldston$^{*}$}

\address{Department of Mathematics and Statistics\\
 San Jos\'{e} State University\\
  California 95192-0103 \\
 United States of America}
\date{\today}
\email{daniel.goldston@sjsu.edu}

\thanks{$^{*}$ Research supported by National Science Foundation Grant DMS-1104434}

\author[Liyang Yang]{ Liyang Yang}

\address{Department of Mathematical Sciences \\
 Tsinghua University \\
 Beijing \\
 100084\\
 P. R. China}

\email{yly12@mails.tsinghua.edu.cn}






\begin{abstract} Assuming the Riemann Hypothesis, we obtain asymptotic formulas for the average number representations of an even integer as the sum of two primes. We use the method of Bhowmik and Schlage-Puchta and refine their results slightly to obtain a more recent result of Languasco and Zaccagnini, and a new result on a smoother average.
\end{abstract}

\maketitle

\thispagestyle{empty}

\section{Introduction} The Goldbach conjecture asserts that every even integer greater than 2 is equal to the sum of two primes. At present there is no proof of this conjecture in sight. Letting $g_2(n)$ denote the number of representations  of the positive integer $n$ as a sum of two primes, then Hardy and Littlewood conjectured that for even $n$
\begin{equation}  g_2(n) := \sum_{n=p+p'}1 \sim  \mathfrak{S}(n) \frac{n}{(\log n)^2}, \qquad \text{as} \quad n \to \infty.\end{equation}
Here $\mathfrak{S}(n)$ is a well-known  arithmetic function called the singular series that plays no role in this paper, although we note that
$\mathfrak{S}(n)=0$ if $n$ is odd, and $\mathfrak{S}(n)>1$ if $n$ is even. Thus we expect that there are many ways to represent a large even number as sums of two primes.

An easier question is to examine the average number of representations of an even number as a sum of two primes. In 1900 Landau \cite{Landau} proved that
\begin{equation} \label{g-average}\sum_{n\le x} g_2(n) \sim \frac12 \frac{x^2}{(\log x)^2}, \quad \text{as} \quad x\to \infty\end{equation}
which is consistent with the conjecture (1.1) since
\[ \sum_{n\le x}\mathfrak{S}(n) \sim x, \quad \text{as} \quad x\to \infty.\]
(This observation is due to Hardy and Littlewood  \cite{HLShahWilson} \cite{HLGoldbach}.)

In 1991 A. Fujii \cite{Fujii1} \cite{Fujii2} \cite{Fujii3} refined \eqref{g-average} considerably and showed that there is a second order term which depends on the zeros of the Riemann zeta-function.  As usual in this field, Fujii worked with a weighted sum for the number of Goldbach representations which also includes powers of prime. Let
\begin{equation} r_2(n) :=\sum_{m+m' =n}\Lambda(m)\Lambda(m').\end{equation}
Then Fujii proved that, assuming the Riemann Hypothesis,
\begin{equation} \label{Fujii} \sum_{n\le x} r_2(n) = \frac12 x^2 -2\sum_{\rho}\frac{x^{\rho +1}} {\rho(\rho+1)} + E_2(x),\end{equation}
where
\begin{equation} \label{AveError} E_2(x) =O((x\log x)^{\frac43}).\end{equation}
Here the sum is over the complex zeros of the Riemann zeta-function and the sum is absolutely convergent.

From work of Granville \cite{Gran1} \cite{Gran2} the question arose of finding the true size of the error term $E_2(x)$. This question was almost completely solved by Bhowmik and Schlage-Puchta \cite{Bhow} who proved on RH that
\[   E_2(x) \ll  x(\log x)^5, \]
and also unconditionally that
\[  \qquad E_2(x) =\Omega ( x\log\log x).\]
This lower bound arises from proving that there exist $n$ for which $r_2(n) > cn\log
\log n$ so that an indivivual term in the average already makes a contribution of this size.\footnote{ This result has been proved before several times, by Prachar \cite{Prachar} in 1951 and Giordano \cite{Giordano} in 2002.}  More recently Languasco and Zaccagnini \cite{Lang-Zac1} have made several contributions to this problem. On RH they improved  the error bound to
\[ E_2(x) \ll x(\log x)^3 .\]
They also recognized that the secondary term arises naturally from the error in the prime number theorem. In \cite{Lang-Zac2} Languasco and Zaccagnini introduced a Ces{\`a}ro weight into the counting formula and proved the remarkable unconditional formula, for $N\ge 2$,
\begin{equation}\label{1.6} \begin{split} \sum_{n\le N } r_2(n)\frac{(1-\frac{n}{N})^k}{\Gamma(k+1)} = \frac{N^2}{\Gamma(k+3)}& - 2\sum_\rho \frac{ \Gamma(\rho)}{\Gamma(\rho +k +2)}N^{\rho+1}  \\ & +\sum_{\rho_1}\sum_{\rho_2}\frac{\Gamma(\rho_1)\Gamma(\rho_2)}{\Gamma(\rho_1+\rho_2+k +1)} N^{\rho_1+\rho_2} + O_k(N^{\frac12}), \end{split} \end{equation} where $k>1$ is a real number. Here the sums are absolutely convergent for $k>\frac12$, and it is reasonable to expect for this formula to hold with possibly a larger error term in this range. Thus the oscillation in $E_2(x)$ (under RH) largely disappears when $r_2(n)$ is counted with a Ces{\`a}ro weight with $k>1$.

\section{Results}

In this paper we follow the method of
Bhowmik-Schlage-Puchta with one modification. This allows us to obtain the same result as Languasco and Zaccagnini.
\begin{theorem}[Languasco Zaccagnini] Assuming the Riemann Hypothesis, we have
\begin{align}\sum_{n\leqslant N}r_2(n)=\frac{1}{2}N^2-\sum_\rho\frac{N^{\rho+1}}{\rho(\rho+1)}+O(N\log^3N).\end{align}
\end{theorem}
Next, we show that if one takes the Ces{\`a}ro average with $k=1$ then we obtain an error term that corresponds to what \eqref{1.6} would imply on RH if one could take $k=1$ in that result.
\begin{theorem} Assuming the Riemann Hypothesis, we have
\begin{equation}\label{eqxxx}  \sum_{n\le N}\left(1 - \frac{n}{N}\right)  r_2(n) = \frac16 N^2 -2\sum_{\rho} \frac{N^{\rho+1}}{\rho(\rho+1)(\rho +2)}  +O(N). \end{equation}
\end{theorem}
A result of this type was mentioned in \cite{Lang-Zac1}.
\section{The Main Terms in the Asymptotic Expansion}
Let $\Lambda_0(n) = \Lambda(n)-1$, where $\Lambda(n)$ is the von Mangoldt function defined by $\Lambda(n)=\log p$ if $n=p^m$, $p$ a prime, $m\ge 1$ an integer, and $\Lambda(n)=0$ otherwise. We always take $n$ and $N$ to be positive integers. Our sums always start at 1 unless specified otherwise. Consider the generating function
\begin{equation}\label{eq1.1} S_0(\alpha,x) = \sum_{n\le x} \Lambda_0(n) e(n\alpha),  \qquad e(u) = e^{2\pi i u}.\end{equation}
Letting $\psi(x)= \sum_{n\le x} \Lambda(n)$, then the prime number theorem takes the form $\psi(x) \sim x$ as $x\to \infty$, and
\begin{equation} \label{eq1.2} \sum_{n\le x}\Lambda_0(n) =\psi(x)-\lfloor x \rfloor = \psi(x) -x +O(1) .\end{equation}
Now
\begin{equation} \label{eq1.3} S_0(\alpha,x)^2 = \sum_{n_1,n_2\le x} \Lambda_0(n_1)\Lambda_0(n_2)e((n_1+n_2)\alpha) = \sum_{n\le 2x} R(n,x) e(n\alpha), \end{equation}
where
\[ R(n,x) = \sum_{\substack{n_1,n_2 \le x\\ n=n_1+n_2}}\Lambda_0(n_1)\Lambda_0(n_2).\]
We can recover $r_2(n)$ from $R(n,x)$ when $n\le x$, since
\begin{equation}\label{1.51}\begin{split} R(n,x) &=  \sum_{ n=n_1+n_2}\Lambda_0(n_1)\Lambda_0(n_2) \\& = r_2(n) -2 \sum_{n=n_1+n_2} \Lambda(n_1) + \sum_{n=n_1+n_2}1 \\& = r_2(n) -E_2(n),\end{split}\end{equation}
where
\begin{equation} \label{1.55} E_2(n) :=  2\psi(n-1) -(n-1)  .\end{equation}
Let
\begin{equation} \label{eq1.5} T(\alpha,N) = \sum_{|n|\le N} t(n) e(n\alpha).\end{equation}
Then,  for $x\ge N$,
\[\int_0^1 S_0(\alpha,x)^2 T(-\alpha,N)d\alpha = \sum_{n\le N} R(n,x)t(n),\]
and by \eqref{1.51} we have obtained the following result.
\begin{lemma}  We have, for any $x \ge N$,
\begin{equation}\label{eq1.6} \sum_{n\le N} \Big( r_2(n) -E_2(n)\Big)t(n) = \int_0^1 S_0(\alpha,x)^2 T(-\alpha,N)d\alpha . \end{equation}
\end{lemma}

For an unweighted average we use
\begin{equation} \label{eq1.7} I(\alpha, N) = \sum_{n\le N} e(n\alpha), \end{equation}
so that in Lemma 1 we take $t(n) = 1$ for $1\le n\le N$ and $t(n)=0$ otherwise.	
\begin{lemma} We have
\begin{equation}\label{eq1.8} \begin{split} \sum_{n\le N}  r_2(n) &= \frac12 N^2 -2\sum_{\rho} \frac{N^{\rho+1}}{\rho(\rho+1)} + \int_0^1 S_0(\alpha,x)^2 I(-\alpha,N)d\alpha\\ & \qquad  - \left(2\log 2\pi -\frac12\right)N +2\frac{\zeta '}{\zeta}(-1) -2\sum_{k=1}^\infty \frac{N^{1-2k}}{2k(2k-1)} ,\end{split} \end{equation}
where $\zeta(s)$ is the Riemann zeta-function and the sum is over the complex zeros $\rho$ of $\zeta(s)$.
\end{lemma}
\begin{proof} To prove \eqref{eq1.8}, we use the function
\begin{equation}\label{eq1.9}  \psi_1(x) := \int_1^x \psi(u)\,du = \sum_{n\le x}(x-n)\Lambda(n),
\end{equation}
which by Theorem 28 of \cite{Ingham1932} or 12.1.1, Exercise 6 of  \cite{MontgomeryVaughan2007} has the explicit formula, for $x \ge 1$,
\begin{equation} \label{eq3.11}\psi_1(x)= \frac12 x^2 -\sum_{\rho}\frac{x^{\rho+1}}{\rho(\rho+1)} - (\log 2\pi)x +\frac{\zeta '}{\zeta}(-1) -\sum_{k=1}^\infty \frac{x^{1-2k}}{2k(2k-1)} , \end{equation}
where the sum is over the zeros of the Riemann zeta-function and is absolutely convergent. From \eqref{1.55} we have
\[\begin{split}\int_0^1 S_0(\alpha,x)^2 I(-\alpha,N)d\alpha &=  \sum_{n\le N}\Big(r_2(n)- 2\psi(n-1) +(n-1)\Big) \\ &= \sum_{n\le N}r_2(n) - 2\psi_1(N)+\frac12 (N-1) N ,\end{split} \]
and substituting the explicit formula for $\psi_1(x)$ proves \eqref{eq1.8}.
\end{proof}
We will also smooth with the Ces{\`a}ro weight by using the Fej{\'e}r kernel
\begin{equation} \label{Fejer} K(\alpha, N) = \sum_{|n|\le N} \left( 1-\frac{|n|}{N}\right)e(n\alpha)= \frac1N\left(\frac{\sin \pi N \alpha}{\sin \pi \alpha}\right)^2. \end{equation}
From Lemma 1 we now obtain the following result.
\begin{lemma} We have
\begin{equation}\label{eq1.12} \begin{split} \sum_{n\le N}\left(1 - \frac{n}{N}\right)  r_2(n) &= \frac16 N^2 -2\sum_{\rho} \frac{N^{\rho+1}}{\rho(\rho+1)(\rho +2)} + \int_0^1 S_0(\alpha,x)^2 K(\alpha,N)d\alpha\\ & \qquad +O(N).\end{split} \end{equation}
\end{lemma}
\begin{proof} The main terms here arise from
\[ \begin{split} \sum_{n\le N} E_2(n)\left(1-\frac{n}{N}\right) &= \frac1{N}\int_{1}^N\sum_{n\le u}E_2(n)\, du \\&
= \frac1{N}\int_{1}^N\sum_{n\le u}(2\psi(n-1) -(n-1))\, du \\&
= \frac1{N}\int_{1}^N \left( 2 \psi_1(u) - \frac12 u^2 +O(u)\right)\, du \\&
=\frac1{N}\int_{1}^N \left( \frac12 u^2 -2 \sum_{\rho}\frac{u^{\rho+1}}{\rho(\rho+1)}+O(u)\right)\, du ,\end{split} \]
and \eqref{eq1.12} follows on substituting \eqref{eq3.11}.

\end{proof}

\section{ Gallagher's Lemma}
Gallagher introduced an important estimate for exponential sums. Let
\[R(\alpha) =\sum_\mu c(\mu) e(\mu \alpha), \quad \text{where} \quad \sum_\mu |c(\mu)| < \infty,\]
and, for $h>0$,
\[C(w,h)=\frac{1}{h}\sum_{\substack{\mu \\ |\mu-w|\le h/2}}c(\mu).\]
Then an easy calculation and Plancherel's theorem gives the identity
\begin{equation}\label{Gallagher} \int_{-\infty}^\infty \left| C(w,h)\right|^2\, dw =
\int_{-\infty}^\infty \left| R(\alpha) \frac{\sin(\pi h \alpha)}{\pi h \alpha}\right|^2 \, d\alpha . \end{equation}
Since $\frac{\sin x}{x}$ is positive and decreasing for $0<x<\pi$, we have
\begin{equation}\label{G-ineq}  \int_{-\frac{1}{2h}}^{\frac{1}{2h}}\left| R(\alpha)\right|^2 \, d\alpha \le \frac{\pi^2}{4}\int_{-\infty}^\infty \left| C(w,h)\right|^2\, dw .\end{equation}
which is one form of Gallagher's lemma.

In our applications we have that $\mu$ is an integer, and thus letting
\[ S(\alpha) = \sum_{n=-\infty}^\infty c_ne(n\alpha), \quad \text{where} \quad \sum_{n=-\infty}^\infty |c_n| < \infty,\]
then  \eqref{Gallagher} and \eqref{G-ineq} become
\begin{equation}\label{Thm1Gallagher} \int_{-\frac{1}{2h}}^{\frac{1}{2h}}\left| S(\alpha)\right|^2 \, d\alpha \ll
\int_{-\infty}^\infty \left| S(\alpha) \frac{\sin(\pi h\alpha)}{\pi h\alpha}\right|^2 \, d\alpha = \frac{1}{h^2}\int_{-\infty}^\infty\left|\sum_{x<n<x+h} c_n\right|^2 \, dx , \end{equation}
which is the inequality we need for Theorem 1. For Theorem 2 we take $h=H$ where $H$ is a positive integer, and since $S(\alpha +1) =S(\alpha)$, we see that
\[ \begin{split} \int_{-\infty}^{\infty} \left| S(\alpha) \frac{\sin(\pi H\alpha)}{\pi H\alpha}\right|^2 \, d\alpha &= \sum_{n=-\infty}^{\infty} \int_n^{n+1}\left|S(\alpha) \frac{\sin(\pi H\alpha)}{\pi H\alpha}\right|^2 \, d\alpha \\ &
= \sum_{n=-\infty}^{\infty} \int_0^{1}|S(\alpha)|^2 \left| \frac{\sin(\pi H(\alpha+n))}{\pi H(\alpha+n)}\right|^2 \, d\alpha \\ &
=  \int_0^{1}\left|S(\alpha)\right|^2 (\sin(\pi H\alpha)^2\sum_{n=-\infty}^{\infty}\frac{1}{\left(\pi H(\alpha+n)\right)^2} \, d\alpha .
\end{split} \]

Since
\[ \left(\frac{\pi}{\sin \pi \alpha}\right)^2= \sum_{n=-\infty}^\infty \frac{1}{(n+\alpha)^2},\]
see Example 10.7 of \cite{Montgomery},
we conclude by  \eqref{Fejer} that
\begin{equation} \int_0^1 |S(\alpha)|^2 K(\alpha,H)\, d\alpha = \frac1 H\int_0^1 |S(\alpha)|^2 \left(\frac{\sin(\pi H\alpha)}{\sin(\pi\alpha)}\right)^2 \, d\alpha = \frac 1 H\int_{-\infty}^\infty\left|\sum_{x<n<x+H} c_n\right|^2 \, dx . \end{equation} This last result, which we use in Theorem 2,  can be checked directly without using Plancherel's theorem.

\section{Bhowmik and Schlage-Puchta's estimate}
We take $S(\alpha)=S_0(x,\alpha)$ in \eqref{Thm1Gallagher} so that $c_n = \Lambda_0(n)$ for $1\le n\le x$ and $c_n=0$ otherwise.
Thus we have, for $1\le h \le x$,
\begin{equation}\begin{split}\label{5.1} \int_{-\infty}^\infty &\left|\sum_{x<n<x+h} c_n\right|^2 \, dx = \int_{-h}^x \bigg|\sum_{\substack{n\le x \\ t<n\le t+h}}\Lambda_0(n)\bigg|^2 \, dt\\ &=
 \int_{-h}^0 \bigg|\sum_{n\le t+h}\Lambda_0(n)\bigg|^2 \, dt +
\int_0^{x-h} \bigg|\sum_{t<n\le t+h}\Lambda_0(n)\bigg|^2\, dt +
\int_{x-h}^x \bigg|\sum_{t<n\le x}\Lambda_0(n)\bigg|^2\, dt .\end{split}\end{equation}
Now by \eqref{eq1.2}
\[  \sum_{a<n\le b}\Lambda_0(n)
= \psi(b)-\psi(a) - (b-a) +O(1),  \]
so that on substituting this into the integrals above and using the inequality  $ (A+B)^2 \le 2A^2 +2B^2$, we obtain the following result of Bhowmik and Schlage-Puchta.

\begin{lemma}[Bhowmik and Schlage-Puchta] We have, for $1\le h \le x$,
\begin{equation}\label{eq1.10} \int_{-\frac{1}{2h}}^{\frac{1}{2h}} |S_0(\alpha,x)|^2\, d\alpha \ll \frac{1}{h^2} \left(I(h) + J(x-h,h) + K(x, h)+x \right), \end{equation}
where
\begin{equation}  I(x) = \int_0^x\left(\psi(t)-t\right)^2\, dt,\end{equation}
\begin{equation} J(x,h) = \int_0^{x}\left(\psi(t+h)-\psi(t)-h\right)^2\, dt,\end{equation}
and
\begin{equation} K(x,h) = \int_{x-h}^x\left(\psi(x)-\psi(t)-(x-t)\right)^2\, dt.\end{equation}
\end{lemma}

\section{ Estimates for Primes on RH}

We see in Lemma 1 from either \eqref{eq1.6} or \eqref{eq1.7} that there isn't actually any dependence on $x$ aside from the condition $x\ge N$. We can remove this dependence explicitly by averaging over $x$ in either formula.  Define an expected value function by
\begin{equation}  E_N(f(\alpha)) = E_N(f(\alpha,x)) := \frac{1}{N}\int_N^{2N}f(\alpha,x)\, dx .\end{equation}
It is immediate that the following result holds.
\begin{lemma} We have
\begin{equation} \int_0^1 S_0(\alpha,x)^2 T(-\alpha,N)d\alpha = \int_0^1 E_N(S_0(\alpha)^2) T(-\alpha,N)d\alpha. \end{equation}
\end{lemma}
Thus we are free to estimate this integral with the average value of $S_0(\alpha, x)^2$ if we choose.
\begin{lemma}  Assuming the Riemann Hypothesis, we have
\begin{equation} \label{cramer} I(x) = \int_0^x\left(\psi(t)-t\right)^2\, dt \ll x^2.\end{equation}
For $1\le h \le x$, we have
\begin{equation}\label{selberg} J(x,h) = \int_0^{x}\left(\psi(t+h)-\psi(t)-h\right)^2\, dt \ll hx(\log \frac {2x}{h} )^2,\end{equation}
and for $1\le h \le N$,
\begin{equation}\label{average} E_N(K(x,h))  = \frac{1}{N} \int_N^{2N} \int_{x-h}^x\left(\psi(x)-\psi(t)-(x-t)\right)^2\, dt\, dx \ll h N.\end{equation}
\end{lemma}
The estimates \eqref{cramer} and \eqref{average} are sharp, while \eqref{selberg} can be improved assuming a pair correlation conjecture. The term $K(x,h)$ depends partly on the error in the prime number theorem from primes up to $x$, and we obtain an improved estimate here by averaging over $x$.

The estimate \eqref{cramer} is due to Cram\'er, see Theorem 13.5 of \cite{MontgomeryVaughan2007}. $J(x,y)$ has been much studied since Selberg \cite{Selberg} first introduced it. The estimate we state is due to Saffari and Vaughan, for a proof see \cite{GoldMont}.  For the last estimate, note the right-hand side is
\[\begin{split} &\ll  \frac{1}{N} \int_N^{2N} \int_{x-h}^x (\psi(x)- x)^2 +(\psi(t)-t)^2\, dt\, dx \\&
\ll  \frac{h}{N}\int_N^{2N} (\psi(x)- x)^2 \, dx + \frac{1}{N}\int_{N-h}^{2N}(\psi(t)-t)^2\int_{t}^{t+h}\, dx\, dt \\&
\ll hN \end{split}\]
on using \eqref{cramer}.

\section{ Proof of Theorems 1 and 2}
Combining the last two lemmas, we obtain the following estimate.
\begin{lemma} Assuming the Riemann Hypothesis, we have for $1\le h \le N$,
\begin{equation} \int_{-\frac{1}{2h}}^{\frac{1}{2h}} |E_N(S_0(\alpha)|^2 d\alpha  \ll \frac{N(\log N)^2}{h} .\end{equation}
\end{lemma}

\begin{proof}[Proof of Theorem 1]

 By Lemma 1 and 2 it suffices to show that
\begin{align}\label{estimate}\int_0^1E_N(|S_0(\alpha)|^2)|I(-\alpha,N)|d\alpha = \int_{-\frac12}^{\frac12}E_N(|S_0(\alpha)|^2)|I(-\alpha,N)|d\alpha\ll N(\log N)^3.\end{align}

We use the estimate $I(\alpha,N) \ll \min(N,\frac{1}{||\alpha ||})$. First, for the range $[-\frac{1}{N},\frac{1}{N}]$
\[ \int_{-\frac1{N}}^{\frac1{N}}E_N(|S_0(\alpha)|^2)|I(-\alpha,N)|d\alpha \ll N \int_{-\frac1{N}}^{\frac1{N}}E_N(|S_0(\alpha)|^2)d\alpha\ll N (\log N)^2 \]
by Lemma 7. We split the rest of the integration range into  dyadic segments, a typical one is
\begin{align*}\int_{\frac{2^k}{N}}^{\frac{2^{k+1}}{N}}E_N(|S_0(\alpha)|^2)|I(-\alpha,N)| d\alpha
&\ll \int_{\frac{2^k}{N}}^{\frac{2^{k+1}}{N}}E_N(|S_0(\alpha)|^2)\frac{ d\alpha}{\alpha}\\
&\ll \frac{N}{2^k}\int_{-\frac{2^{k+1}}{N}}^{\frac{2^{k+1}}{N}}E_N(|S_0(\alpha)|^2) d\alpha \\
&\ll N(\log N)^2\end{align*}
 by Lemma 7.   Because there are $O\left(\log N\right)$ such $k$, we obtain \eqref{estimate}.
\end{proof}

\begin{proof}[Proof of Theorem 2] By \eqref{Fejer}, \eqref{5.1} and Lemma 4 we have
\[ \int_0^1 |S_0(\alpha,x)|^2 K(\alpha,N)\, d\alpha \ll \frac{I(N) + J(x-N,N)+K(x,N)} N ,\]
and therefore by Lemmas 5 and 6
\[ \int_0^1 |S_0(\alpha,x)|^2 K(\alpha,N)\, d\alpha \ll N.\]
The result now follows from Lemma 3.
\end{proof}

\end{document}